\documentclass[12pt]{amsart}

\usepackage{amsmath} 
\usepackage{amsthm} 
\usepackage{amsfonts} 
\usepackage{amssymb}
\usepackage{longtable}
\usepackage{geometry}
\usepackage{tikz}
\usepackage{pgflibraryshapes}

\usetikzlibrary{arrows}
\usetikzlibrary{plothandlers}
\usetikzlibrary{calc}

\numberwithin{equation}{section}

\newcommand\hf{\hfill}
\newcommand\ds{\displaystyle}
\renewcommand\({\left(}
\renewcommand\){\right)}

\newcommand\ganz{\mathbb Z} 
\newcommand\real{\mathbb R}

\newcommand\pol{\mathcal P}

\newcommand\conv{\text{\rm Conv}}

\newcommand\stab{\text{\rm Stab}}
\newcommand\ZT{\text{\rm ZT}}
\newcommand\ZN{\text{\rm Zon}}
\newcommand\supp{\textrm{Supp}}

\newcommand\0{\underline 0}

\newtheorem*{thm*}{Theorem}
\newtheorem{thm}{Theorem}[section]
\newtheorem{pro}[thm]{Proposition}
\newtheorem{lem}[thm]{Lemma}

\newtheorem*{cor*}{Corollary}

\theoremstyle{definition}
\theoremstyle{definition}
\theoremstyle{definition}
\theoremstyle{definition}

\begin{document}

\title{Polar Root Polytopes that are Zonotopes}

\author{Paola Cellini}

\address{Paola Cellini\\ Dipartimento di Ingegneria e Geologia\\ Universit\`a di Chieti e Pescara\\ Viale Pindaro 42, 65127 Pescara PE, Italy}

\email{pcellini@unich.it}

\author{Mario Marietti}

\address{Mario Marietti\\ Dipartimento  di Ingegneria Industriale e Scienze Matematiche\\ Universit\`a Politecnica delle Marche\\ Via Brecce Bianche, 60131 Ancona AN,  Italy}

\email{m.marietti@univpm.it}


\begin{abstract}
Let $\pol_{\Phi}$ be the  root polytope  of a finite irreducible crystallographic root system $\Phi$, i.e., the convex hull of all  roots in $\Phi$. The polar of  $\pol_{\Phi}$, denoted $\pol_{\Phi}^*$, coincides with the union of the orbit of the fundamental alcove under the action of the Weyl group. In this paper, we establishes which polytopes   $\pol_{\Phi}^*$ are zonotopes and which are not. The proof is constructive.
\end{abstract}

\maketitle

\section{Introduction}
Let $\Phi$ be a finite irreducible crystallographic root system in a Euclidean space $V$ with scalar product $( \ \, , \ )$, $W$  the Weyl group of $\Phi$, and $\mathfrak g_{\Phi}$ a simple Lie algebra having $\Phi$ as root system. 
Let $\pol_{\Phi}$ be the {\it root polytope} associated with $\Phi$, i.e. the convex hull of all  roots in $\Phi$. 

\par
Motivated by the connections of the root polytope $\pol_{\Phi}$ with $\mathfrak g_{\Phi}$ (more precisely, with the Borel subalgebras of $\mathfrak g_{\Phi}$ and their abelian ideals), in \cite{CM1} we study $\pol_{\Phi}$ for a general $\Phi$.
Among other things, we give a presentation of  $\pol_{\Phi}$  as an intersection of half-spaces, and describe its faces as special subposets of the root poset, up to the action of $W$.
In~\cite{CM2}, we develop these general results, obtaining several special results for the root types $\mathrm A_n$ and $\mathrm C_n$.
One of the special properties of these two root types is that the cones on the facets of $\pol_{\Phi}$  are the closures of the regions of a hyperplane arrangement. This means that $\pol_{\Phi}$ is combinatorially dual to a zonotope (see \cite[\S 7.3]{Zieg}, or \cite[\S 2.3.1]{DecoPro}).
More precisely, if  $\mathcal H_{\pol_{\Phi}}$ is the arrangement of all the hyperplanes through  the origin containing some ($n-2$)\nobreakdash-dimensional faces of $\pol_{\Phi}$, then the complete fan associated to $\mathcal H_\pol$ is equal to the face fan associated to  $\pol_{\Phi}$. This property is satisfied  by the polytopes whose polar polytopes are  zonotopes (see \cite[\S 7.3]{Zieg}).
One of the referees of \cite{CM2} asked if the polars of  the types $\mathrm A_n$ and $\mathrm C_n$  root polytopes are actually zonotopes, and if so which.

\par
In this paper, we  answer this question. We denote by $\pol^*_{\Phi}$ the polar polytope of $\pol_{\Phi}$: we explicitly describe $\pol^*_{\Phi}$ as a zonotope  for the types $\mathrm A_n$ and $\mathrm C_n$, as well as for $\mathrm B_3$ and $\mathrm G_2$.  Moreover, we prove by a direct check that, for all other root types, the set of cones on the facets of $\pol_{\Phi}$ is not equal to the set of closures of the regions of $\mathcal H_{\pol_{\Phi}}$. Hence, for all other root types, $\pol^*_{\Phi}$ is not  a zonotope.

\par
We point out  that $\pol^*_{\Phi}$ is a natural object for the crystallographic root system $\Phi$ that can be more familiarly described in terms of alcoves and Weyl groups. 
Indeed, $\pol^*_{\Phi}$ is the union of the orbit of a fundamental alcove of $\Phi$ under the Weyl group, i.e., if we fix any basis of $\Phi$, and denote by $\mathcal A$ the corresponding fundamental alcove of the affine Weyl group associated to $\Phi$, then $\pol^*_{\Phi}=\bigcup_{w\in W} w\mathcal A$. 
Thus, $\pol^*_{\Phi}$ is a fundamental domain for the group of translations by elements in the coroot lattice of $\Phi$.

\section{Statement of results}\label{statement}

Let  $\Phi^+$ be a positive system for $\Phi$, $\Pi$ the corresponding root basis of $\Phi$, $\theta$ the highest root, and $\Omega^\vee$ the dual basis of $\Pi$ in the dual space $V^*$ of $V$, i.e. the set of fundamental coweights of~$\Phi$.  

\par 
We set \ $\Pi=\{\alpha_1, \dots, \alpha_n\}$, \ $\theta=\sum_{i=1}^nm_i\alpha_i$,  \ $\Omega^\vee=\{\omega^\vee_1, \dots, \omega^\vee_n\}$, so that  $\langle \omega^\vee_i ,  \alpha_j \rangle=\delta_{ij}$, where $\langle \ \, , \ \rangle: V^* \times V \mapsto \mathbb R$ is the natural pairing of $V^*$ and $V$,  and define 
$$
o_i=\frac{\omega^\vee_i}{m_i}, \text{ for } i=1, \dots, n.$$
Consider the  fundamental alcove of the affine Weyl group of~$\Phi$ (see \mbox{\cite[VI, 2.1--2.2]{Bou}}, or \mbox{\cite [4.2--4.3]{Hum}}):
$$
\mathcal A=\{x\in V^*\mid \langle x, \alpha \rangle \geq 0 \text{ for all } \alpha\in \Pi,\ \langle x, \theta \rangle \leq 1\}.
$$
If $x$ is any element or subset of $V$ or $V^*$, we denote by $W\cdot x$ the orbit of $x$ by the action of $W$. It is well-known, and easy to see, that $\mathcal A$ is the $n$-simplex with vertices the null vector and $o_1, \dots, o_n$ \cite[VI, Corollaire in 2.2]{Bou}, and that 
\begin{equation}\label{A}
\bigcup_{w\in W} w  \mathcal A =\{x\in V^* \mid \langle x, \beta  \rangle \leq 1 \text { for all } \beta\in \Phi\}.
\end{equation}
We denote by $\pol_{\Phi}$  the {\it root polytope} associated to $\Phi$, i.e. the convex hull 
of all roots in $\Phi$. For short, we write   $\pol$ for $\pol_{\Phi}$ when the root system is clear from the context.
It is easy to see that~$\pol$ is the convex hull of the long roots in~$\Phi$. 
Indeed, this is directly checked if the rank of~$\Phi$ is $2$. In the general
case, we observe that, since $\Phi$ is irreducible, the set of all  long roots and the set of all short roots cannot be mutually  orthogonal, hence there exist a short root $\beta$ and a long root  $\beta'$ such that $(\beta, \beta')\neq 0$. Then, $\beta$ belongs to  the convex hull of the long roots in the irreducible dihedral root system  generated by $\beta$ and $\beta'$, and since $W$ is transitive on the  short roots, any short root belongs to the convex hull of some long  roots.

We denote by  $\pol_{\Phi}^*$ the polar of the root polytope associated with $\Phi$: 
$$
\pol_{\Phi}^*:= \{ x \in V^* \mid  \langle x, v  \rangle \leq 1 \text{ for all $v \in \pol_{\Phi}$}  \},
$$
and we call it the {\it polar root polytope}. Again, for short, we write $\pol^*$ for $\pol_{\Phi}^*$. 
By definition of $\pol^*$, we obtain
\begin{equation}\label{P}
\pol^*=\{x\in V^*\mid   \langle x, \beta  \rangle \leq 1 \text { for all } \beta\in \Phi\}=
\{x\in V^*\mid  \langle x, \beta  \rangle \leq 1 \text { for all } \beta\in \Phi_\ell\},
\end{equation}
where $\Phi_\ell$ is the set of long roots in $\Phi$. 
Hence, the polar root polytope satisfies
\begin{equation}\label{PA}
\pol^*=\bigcup_{w \in W} w \mathcal A=\conv\(\bigcup_{i=1}^n W\cdot o_i\).
\end{equation}

\par
Our first result is that, for $\Phi$ of type $\mathrm A_n$, $\mathrm C_n$ (hence also for type $B_2=C_2$),  $B_3$,  and $G_2$,  the polar root polytope $\pol^*$ is a zonotope.

\par
A zonotope  is by definition the image of a cube under an affine projection. 
Let $U$ be a vector space, $v_1, \dots, v_k, p\in U$, and $S=\{v_1, \dots, v_k\}$. 
We set
$$
\ZN_p (S)= \left\{p + \sum\limits_{i=1}^k t_iv_i\mid  - \frac{1}{2}\leq t_i\leq  \frac{1}{2}\right\}
$$
and call $\ZN_p (S)$ the {\em zonotope generated by $S$ with center $p$}. Thus 
a zonotope in $U$ is a polytope of the form  $\ZN_p (S)$, for some finite subset $S$ and some vector $p$ in $U$.
\par
We prove that, for $\mathrm A_n$ and $\mathrm C_n$, the polar root polytope $\pol^*$ is the zonotope generated by the orbit of a single $o_i$ (with center the null vector $\0$),  and that a similar result holds for $\mathrm B_3$ and $\mathrm G_2$. This is done  in Section~\ref{are-zone}, where we find case free conditions for the general inclusions $\ZN_{\0} (W\cdot c\,\omega_i^\vee)\subseteq \mathcal P^*$ ($c\in \real$, $i\in \{1, \dots, n\}$), and check directly the reverse inclusions that we need in the special cases $A_n, C_n, B_3, G_2$.
\par
The following is a well-known property of the zonotopes (see  \cite[Corollary~7.18]{Zieg}).

\begin{pro}
\label{prop-zon}
Let $U$ be a vector space, $U^*$ its dual,  $S=\{v_1, \ldots, v_r\}\subseteq U\setminus\{\underline 0\}$, 
$H_i = \{v \in U^* \mid \langle v, v_i \rangle =0 \}$ for $i=1, \dots, r$, and 
$\mathcal A_S$ be the arrangement of the hyperplanes $H_1, \ldots, H_r$. 
Then the cones on the faces of $\ZN_0(S)^*$  coincide with the faces of $\mathcal A_S$.
\end{pro}

We denote by $\mathcal H_\pol$ the central hyperplane arrangement determined by the $(n-2)$-faces of $\pol$, i.e., $H\in \mathcal H_\pol$ if and only if $H$ is a hyperplane containing the null vector $\0$ and some $(n-2)$ face of $\pol$. Since the null vector $\0$ lies in interior of $\pol$, the polytopes $\pol$ and $\pol^*$ are combinatorially dual to each other, and  $\pol = (\pol^*) ^*$.  
By Proposition~\ref{prop-zon}, if $\pol^*$ is a zonotope, then the cones on the proper faces of $\pol$ should coincide with the faces of  the hyperplane arrangement $\mathcal H_\pol$. Therefore, $\pol^*$ cannot be a zonotope if some hyperplane in $\mathcal H_\pol$ meets the interior of some facet of $\pol$. We will see in Section~\ref{are-not-zone} that this happens for all root types other than $\mathrm A_n$, $\mathrm C_n$, $\mathrm B_3$, $\mathrm G_2$. 

\par
We sum up our results in the following theorem, where we number the simple roots and hence the fundamental coweights as in Bourbaki's tables  \cite{Bou}.  For any finite subset $S=\{v_1, \dots, v_k\}$, we denote by $\ZT$ the following zonotope:
$$\ZT (S)= \left\{\sum\limits_{i=1}^k t_iv_i\mid  0 \leq t_i\leq 1 \right\}.$$
If the barycenter of $S$ is the null vector $\0$, then $\ZT(S)$ coincides with $\ZN_{\0}\(S\)$ (see Section~4).

\begin{thm}\label{main}
\begin{enumerate}
\item{}
For $\Phi$ of type $\mathrm A_n$ or $\mathrm C_n$,\
$$\pol^*=\ZT(W\cdot o_1).$$
\item{} For $\Phi$ of type $\mathrm B_3$, 
$\pol^*=\ZT\(W\cdot \ds\frac{o_3}{2}\)$.
\item
For $\Phi$ of type $\mathrm G_2$, 
$\pol^*=\ZT\(W\cdot \ds\frac{o_1}{2}\)$.
\item{}
For all other root types, $\pol^*$ is not  a zonotope.
\end{enumerate}
\end{thm}

\section{Preliminaries}

In this section, we set our further notation and collect some basic results on root systems and Weyl  groups. Some of these results are well-known (see \cite{BB}, \cite{Bou}, or \cite{Hum}), while other results are more unusual and their proofs will be sketched.

\par
Let $\Phi$ be a finite irreducible (reduced) crystallographic root system in the
real 
vector space $V$ endowed with the positive definite
bilinear form $(\ \, , \ )$. From now on, for notational convenience,  we identify $V^*$ with $V$ through the  form  $(\ \, , \ )$.

We sum up our notation on the root system and its Weyl group in the following list:
\smallskip 

{\renewcommand{\arraystretch}{1.2}
$
\begin{array}{@{\hskip-1.3pt}l@{\qquad}l}
n &  \textrm{the rank of $\Phi$}, 
\\
\Pi= \{\alpha_1, \ldots, \alpha_n\} &  \textrm{the set of simple roots}, 
\\
\Omega^\vee=\{\omega^\vee_1, \dots, \omega^\vee_n\}  &  \textrm{the set of fundamental coweights (the dual basis of $\Pi$)},
\\
\Phi^+  &  \textrm{the set of positive roots w.r.t. $\Pi$},
\\
\Phi(\Gamma) &  \textrm{the root subsystem generated by $\Gamma$ in $\Phi$, for $\Gamma\subseteq \Phi$,}
\\
\theta  & \textrm{the highest root in $\Phi$},
\\
m_i &   = \textrm{the $i$-th coordinate of $\theta$ w.r.t. $ \Pi$, i.e. $\theta=\sum_{i=1}^nm_i \alpha_i$},
\\
o_i & = \frac{\omega^\vee_i}{m_i}, \text{ for } i=1, \dots, n,
\\
\Phi^\vee &=  \{\beta^{\vee}= \frac {2\beta}{(\beta, \beta)} \mid \beta \in \Phi\}, \textrm{ the dual root system of  $\Phi$},
\\

W   &  \textrm{the Weyl group of $\Phi$},
\\
s_{\beta}   & \textrm{the reflection with respect to the root $\beta$}.
\end{array}$
\bigskip

For each specific type of the root system,  we number the simple roots and hence the fundamental coweights as in Bourbaki's tables \cite{Bou}.

\subsection{Root, coroots and partial orderings.} 
We denote by $\leq$ the usual partial ordering of $V$ determined by the positive system $\Phi^+$: $x\leq y$ if and only if $y-x=\sum_{\alpha\in \Pi} c_\alpha\alpha$ with $c_\alpha$ a nonnegative integer for all $\alpha\in \Pi$. We denote by $\leq^\vee$ the analogous ordering determined by the dual root system $\Phi^\vee$.
For $S\subseteq\Phi$,  we let $S^\vee=\{\beta^\vee\mid \beta\in S\}$. Then $\Pi^\vee$ is the basis of  
$\Phi^\vee$ corresponding to the positive system $(\Phi^+)^\vee$. We have: 
$$
x\leq^\vee y\quad\text{if and only if}\quad y-x=\sum_{\alpha\in \Pi} c_\alpha\alpha^\vee \quad \text{with } c_\alpha\in \ganz,\ c_\alpha\geq 0 \quad\text{for all }\alpha\in \Pi.
$$

\subsection{Reflections products.} 
In the following lemma, we provide a result on the product of reflections by general roots (possibly not simple).
\begin{lem}
\label{prodotti}
Let $\beta_1, \dots, \beta_k\in \Phi$
and $w=s_{\beta_1}\cdots s_{\beta_k}$. Then, for any $x\in V$, 
\begin{equation}\label{base}
w(x)=x- \sum\limits_{i=1}^k (x, \beta_i^\vee)\nu_i=x- \sum\limits_{i=1}^k (x, \beta_i)\nu_i^\vee,\quad
\text{where}
\quad
\nu_i=s_{\beta_1}\cdots s_{\beta_{i-1}} (\beta_i);
\end{equation}
\begin{equation}\label{duale}
w(x)=x- \sum\limits_{i=1}^k (x, \eta_i^\vee)\beta_i=x- \sum\limits_{i=1}^k (x, \eta_i)\beta_i^\vee,\quad
\text{where}
\quad
\eta_i=s_{\beta_k}\cdots s_{\beta_{i+1}}(\beta_i).
\end{equation} 
\end{lem}
\sloppy

\begin{proof}
The first equality in Formula~\eqref{base} is easily proved by induction computing $s_{\beta_1} (s_{\beta_2} \cdots s_{\beta_k }(x))$; the second one is clear since $\beta_i$ and $\nu_i$ have the same length, for all~$i$. 
Formula~\eqref{duale} is an application of \eqref{base} since, by definition of $\eta_i$ and the fact that $s_{s_{\beta}(\beta')}=s_{\beta} s_{\beta'} s _{\beta}$ 
for all roots $\beta$ and $\beta'$,  we have
$$
s_{\beta_h}\cdots s_{\beta_k}=s_{\eta_k}\cdots s_{\eta_h},  \quad
\textrm{for\ } h=1, \dots, k,
$$ 
hence 
$w=s_{\eta_k}\cdots s_{\eta_1}$ and $\beta_i=s_{\eta_k}\cdots s_{\eta_{i+1}} (\eta_i)$.
\end{proof}

\fussy

For $h=1, \dots, k$, if we define $w_h=s_{\beta_h}\cdots s_{\beta_k}=s_{\eta_k}\cdots s_{\eta_h}$, Formula~\eqref{duale}  yields
\begin{equation}\label{dualeh}
w_h(x)=x- \sum\limits_{i=h}^k (x, \eta_i^\vee)\beta_i=x- \sum\limits_{i=h}^k (x, \eta_i)\beta_i^\vee.
\end{equation}

\subsection{Reduced expressions.}
For any $w\in W$, we set 
$$N(w)=\{\gamma\in \Phi^+\mid w^{-1}(\gamma)<\0\}.$$
If  $w=s_{\beta_1}\cdots s_{\beta_k}$  is a reduced expression of $w$ (so $\beta_i \in \Pi$, for all $i =1, \ldots, k$), and if we 
 define $\nu_i=s_{\beta_1}\cdots s_{\beta_{i-1}} (\beta_i)$ and $\eta_i=s_{\beta_k}\cdots s_{\beta_{i+1}}(\beta_i)$ for $i=1, \ldots, k$ as in Lemma~\ref{prodotti},  then we have the following result (see \cite[VI, 1.6, Corollaire 2]{Bou}):
\begin{equation}\label{Nw}
N(w)=\{\nu_1, \dots, \nu_k\}\quad \text{ and } \quad
N(w^{-1})=\{\eta_1, \dots, \eta_k\}.
\end{equation}

\subsection{Stabilizers.}
We denote by $\mathcal C$ the fundamental chamber of $W$: 
$$
\mathcal C=\{x\in V\mid (x, \alpha)\geq 0 \text{ for all } \alpha\in \Pi\}.
$$
Formula~\eqref{base} implies directly the well-known fact that
\begin{equation}\label{stabdom}
\stab_W(x)=\langle s_{\alpha}\mid \alpha\in\Pi,\ (x, \alpha)=0\rangle,  \quad 
\text{for all } x\in \mathcal C
\end{equation}
and, as an easy consequence, the equally well-known fact that
\begin{equation}\label{stabgen}
\stab_W(x)=\langle s_\beta\mid \beta\in \Phi,\  (x, \beta)=0\rangle,\quad \text{for all } x\in V.
\end{equation}

\par
For all $j\in\{1, \dots, n\}$, we set 
$$
W^j=\langle s_{\alpha_i}\mid i\in [n]\setminus\{j\}\rangle,
$$
so that 
$$
W^j=\stab_W(\omega_j^\vee).
$$

\subsection{Images of fundamental coweights.}
Let $w=s_{\beta_1}\cdots s_{\beta_k}$ be a reduced expression of~$w$ and $\eta_i=s_{\beta_k}\cdots s_{\beta_{i+1}}(\beta_i)$, for $i=1, \ldots, k$.
Recall that, by definition, the left descents of $w$ are the simple roots in $N(w)$, and the right descents  the simple roots in $N(w^{-1})$.  
For any $j\in \{1, \dots, n\}$, if $w$ is a minimal length representative in the left coset $wW^j$, then $\alpha_j$ is the unique right descent of $w$. Hence, every reduced expression of $w$ ends with $s_{\alpha_j}$ and every reduced expression of $w^{-1}$ starts with $s_{\alpha_j}$, i.e., in our notation, 
\begin{equation}\label{discesa}
\beta_k=\eta_k=\alpha_j.
\end{equation} 
For $\gamma\in\Phi$, let 
$$
\supp(\gamma)=\{\alpha_i\in \Pi\mid (\gamma, \omega_i^\vee)\neq 0\}
$$ 
and, for all $\alpha\in \Pi$, let 
$$
M_{\alpha}=\{\gamma\in \Phi^+\mid \alpha\in \supp(\gamma)\}.
$$
It is clear that if $\supp(\gamma)\cap N(w^{-1})=\emptyset$, then $\gamma\not\in N(w^{-1})$, by the linearity of $w$. Hence, 
if $w$ is the minimal length representative in $wW^j$, we have
\begin{equation}\label{cosetminimale}
\{\eta_1, \dots, \eta_k\}\subseteq M_{\alpha_j}.
\end{equation}
Equivalently,  $(\omega^\vee_j, \eta_i)\geq 1$ for $i=1, \dots, k$. Hence, by \eqref{duale} and \eqref{dualeh}, we have
\begin{equation}\label{cosetminimale2}
w_h(\omega_j^\vee)\leq^\vee \omega_j^\vee-\beta_k^\vee \dots -\beta_h^\vee.
\end{equation}
In particular, by \eqref{discesa}, for each $w\not\in \stab_W(\omega_j^\vee)$
\begin{equation}
w (\omega_j^\vee) \leq^\vee \omega_j^\vee-\alpha_j^\vee.
\end{equation}

\section {Polar root polytopes that are zonotopes} 
\label{are-zone}

In this section, we prove items (1), (2), and (3) of Theorem~\ref{main}.

\par
By Proposition~\ref{prop-zon}, if $\pol^*$ is a zonotope, then  the cones on the proper faces of $\pol$  coincide with the faces of the hyperplane arrangement $\mathcal H_{\pol}$, and $\pol^*= \ZN_{\0}(S)$, where $S$ is a complete set of orthogonal vectors to the hyperplanes of  $\mathcal H_{\pol}$. 
Recall from \mbox{\cite[Proposition~3.2]{CM2}} that the hyperplanes in the arrangement $\mathcal H_{\pol}$ are of a very special form: there exists a subset $H_{\Phi}\subseteq \{1, \ldots, n\}$ (depending on $\Phi$) such that 
$\mathcal H_{\pol}= \{w(\omega^\vee_k)^\perp \mid w \in W , k \in H_{\Phi}\}$. The sets $H_{\Phi}$ are given in   \cite[Table 2]{CM2}. We will see that, when  $\pol^*$ is a zonotope, the set $S$ generating it is the $W$-orbit of a multiple of a single coweight.

\par
>From the definitions (see Section~2), it is clear that
$\ZT(S)= \ZN_p (S),$
with \mbox{$p= {\frac{1}{2}\sum_{s \in S} s }$}. For any $\lambda \in V$, we get
$$
\ZT(W \cdot \lambda) = \ZN_{\0} (W \cdot \lambda),
$$
since ${\frac{1}{2}\sum_{w \in W} w (\lambda) } $ is fixed by all elements in $W$ and so must be the null vector $\0$. 

\par
To prove items (1), (2), and (3) of Theorem~\ref{main}, we need the following lemmas.
\begin{lem}
\label{contiene-contenuto} Let $S$ be a $W$-stable finite subset of $V$. 
\begin{enumerate}
\item
\label{iene}
$\ZT(S)\subseteq \pol^*$ if and only if, for each $X\subseteq  S$,  $\(\sum_{x\in X} x, \theta\)\leq 1$.
\item
\label{iuto}
 If for each $i\in \{1, \dots, n\}$ there exists $X\subseteq S$ such that $\sum_{x\in X} x=o_i$, then \mbox{$\pol^*\subseteq \ZT(S)$}.
\end{enumerate}
\end{lem}

\begin{proof}
(1) It is easy to see that the set of vertices of $\ZT(S)$ is a subset of $\left\{\sum_{x\in X} x \mid 
X\subset S\right\}$ (see for example \cite[\S 2.3]{DecoPro}). Hence, the claim follows from 
\eqref{P} and  the stability of $S$ under $W$, since all long roots are in the same $W$-orbit.
\par\noindent
(2) By \eqref{PA} and the stability of $S$ under $W$, the assumption that for each $i\in \{1, \dots, n\}$,  $o_i=\sum_{x\in X} x$, with  $X\subseteq S$, implies that the  set of vertices of $\pol^*$ is contained in the set  of vertices of $\ZT(S)$, hence the claim.
\end{proof}

\bigskip
For each $j\in \{1, \dots, n\}$,  we set 
$$
r_j=\frac{\|\theta\|^2}{\|\alpha_j\|^2}.
$$

\begin{lem}\label{er}
Let  $j\in \{1, \dots, n\}$.
For each long root $\beta$, we have 
$$(\beta, \omega_j^\vee)\equiv m_j\mod r_j.$$ 
\end{lem}

\begin{proof}
The claim is obvious if $\alpha_j$ is long, in which case $r_j=1$. 
We recall that if $\alpha_j$ is short, then $(\gamma, \alpha_j^\vee)\in \{-r_j, 0, r_j\}$  for each long root $\gamma$.   
Hence the claim follows by induction on the length of $w\in W$ such that $w(\theta)=\beta$. 
\end{proof}

For each subset $S$ of $\Pi$, we denote by $\Phi(S)$ the standard parabolic subsystem of $\Phi$ generated by $S$, and by $W(S)$ the Weyl group of $\Phi(S)$. 
We set 
$$
\Phi_0=\Phi(\Pi\cap \theta^\perp), \qquad W_0=W(\Phi_0).
$$ 
For each  $\alpha\in \Pi$, $\alpha\perp \theta$ if and only if $\alpha$ is not connected to $\alpha_0$, the extra root added to $\Pi$ in the extended Dynkin diagram of $\Phi$  (see \cite [\S  4.7]{Hum}, or \cite[Chapter VI, n$^\text{o}$ 4.3]{Bou}).  Since $\theta$ is 
in the fundamental chamber of $W$,  
$$
W_0=\stab_W(\theta).
$$

\par
For each  $j\in \{1, \dots, n\}$, we set 
$$
\Phi_0^j=\Phi((\Pi\cap \theta^\perp)\setminus \{\alpha_j\}), \qquad W_0^j=W(\Phi_0^j), \qquad q_j=[W_0: W_0^j].
$$  
It is clear that  $$W(\Phi_0^j)=\stab_{W}(\omega_j^\vee)\cap\stab_W(\theta)=\stab_{W_0}(\omega_j^\vee).
$$ 

\begin{lem}\label{ef}
Let $j\in \{1, \dots, n\}$ and $x\in W\cdot \omega_j^\vee$. Then,
$\left(x, \theta\right)=m_j$ if and only if  $x\in W_0\cdot \omega_j^\vee$. In particular, 
$$
\left|\left\{x\in W\cdot \omega_j^\vee\mid \left(x, \theta\right)=m_j\right\}\right|=q_j.
$$
\end{lem}

\sloppy
\begin{proof}
It is obvious that,  if $w\in W_0$, then  $\left(w(\omega^\vee_j), \theta\right)=m_j$, since $\left(w(\omega^\vee_j), \theta\right)= \left(\omega^\vee_j, w^{-1}(\theta)\right)=\left(\omega^\vee_j, \theta\right)$. Conversely,  assume  $x\in W\cdot \omega_j^\vee$ and  $\left(x, \theta\right)=m_j$. Let $w$ be the minimal length element in  $W$ such that $x=w(\omega_j^\vee)$, and $w=s_{\beta_1}\cdots s_{\beta_k}$ be a reduced expression of $w$. Then, by \eqref{cosetminimale2},
$x\leq^\vee \omega^\vee_j- \sum\limits_{i=1}^k \beta_i^\vee$, and hence,
since  $\theta$ is in the fundamental chamber and $(\omega_i^\vee, \theta)=m_i$, we obtain $(\theta,\beta^\vee_i)=0$ for $i=1,\dots, k$, i.e. $w\in W_0$.
\end{proof}
\fussy

\begin{pro}\label{maxZT}
Let $j\in  \{1, \dots, n\}$. If $r_j=m_j$, then for all $c\in \real$, 
$$\ZT\(W\cdot c\,\omega_j^\vee\)\subseteq \pol^*
\text{ if and only if  }\ c\leq \frac{1}{q_j m_j}.$$
Equivalently, if $r_j=m_j$, then the hyperplane $\{(x, \theta)=q_j m_j\}$ is a supporting hyperplane for $\ZT\(W\cdot \omega_j^\vee\)$. 
\end {pro}

\begin{proof}
By Lemma~\ref{ef},   for all $w\in W\setminus W_0$, 
 $(w(\omega_j^\vee), \theta)< m_j$, hence  by Lemma~\ref{er}, $(w(\omega_j^\vee), \theta)\leq 0$. 
It follows that  
$$\(\sum\limits_{x\in W_0\cdot \omega_j^\vee}x, \theta\)=q_j m_j,$$
and that, for any other $z\in \ZT\(W\cdot \omega_j^\vee\)$, $(z,\theta)\leq q_j m_j$. 
This proves the claim.
\end{proof}

For both types $\mathrm A_n$ and $\mathrm C_n$, $\mathcal H_{\pol}=  \{w(\omega^\vee_1)^\perp \mid w \in W \}$, and  the values of $\Phi_0$, $\Phi_0^1$, $q_1$,  $m_1$, and $r_1$ are the following:

\bigskip
$\begin{matrix}
\mathrm A_n: &\Phi_0=\Phi_0^1=\Phi(\Pi\setminus\{\alpha_1, \alpha_n\}),\quad  q_1=m_1=r_1=1,
\end{matrix}$

\bigskip
$\begin{matrix}
\mathrm C_n:  &\Phi_0=\Phi_0^1=\Phi(\Pi\setminus\{\alpha_1\}),\quad q_1=1,\quad
m_1=r_1=2.
\end{matrix}$

\bigskip
\noindent
By Proposition~\ref{maxZT}, from the previous computations, in types $\mathrm A_n$ and $\mathrm C_n$, we have 
$\ZT\(W\cdot o_1\) \subseteq \pol^*$.
In both cases, also the other inclusion holds. 
\begin{thm}
Let $\Phi$ be of type $\mathrm A_n$ or $\mathrm C_n$. Then
$$
\ZT\(W\cdot o_1\) = \pol^*.
$$
\end{thm}

\begin{proof}
To prove that $\ZT\(W\cdot o_1\) \supseteq \pol^*$,
we show that, for each  $k\in \{1, \ldots, n\}$, $o_k$ is a sum of distinct elements  in $W\cdot o_1$ (Lemma~\ref{contiene-contenuto}, (\ref{iuto})). 

\par
Let $w_i:=s_i\cdots s_1$, for $i\in \{1, \ldots, n\}$, and 
$w_0:=e$, the identity element. Let us show by  induction on  $k$, $0\leq k< n$,  that $\sum\limits_{i=0}^{k} w_i (o_1)=o_{k+1}$. This is trivially true  for  $k=0$. If $k>0$ and $\sum\limits_{i=0}^{h} w_i (o_1)=o_{h+1}$,  $h<k<n$, then $w_h (o_1)= -o_{h}+o_{h+1}$ (where  $o_0:=\underline 0$).  We have 
$$
\sum\limits_{i=0}^{k} w_i (o_1)=o_{k}+w_k(o_1)=
o_{k}+s_k w_{k-1}(o_1)= o_{k}+s_k (-o_{k-1}+o_k)=o_k-o_{k-1}+s_k o_k.
$$ 
Since $(o_i, \alpha_j)=\frac{1}{m_i}\delta_{i,j}$,  we have
$(s_k (o_k), \alpha_j)=(o_k, s_k(\alpha_j))=
\frac{-(\alpha_k^\vee,\alpha_j)}{m_k}$,
and hence
$$
\left(\sum\limits_{i=0}^{k} w_i (o_1), \alpha_j\right)=0
$$ 
for  $j=k$ and for all $j$ with $|j-k|>1$. Only the two cases $j=k-1$ and $j=k+1$ are left out. We have 
$$
\left(\sum\limits_{i=0}^{k} w_i (o_1), \alpha_{k-1}\right)=-\frac{1}{m_{k-1}}-\frac{(\alpha_k^\vee, \alpha_{k-1})}{m_k}=0, 
$$  
(note that, for type $\mathrm C_n$, we need $k<n$), and 
$$
\left(\sum\limits_{i=0}^{k} w_i (o_1), \alpha_{k+1}\right)=
-\frac{(\alpha_k^\vee, \alpha_{k+1})}{m_k}=1.
$$ 
So we get the assertion.
\end{proof}

While the property $\ZT\(W\cdot o_1\) = \pol^*$ is a property of the root system $\Phi$, the property that  $ \pol^*$ is a zonotope is a property of the root polytope $\pol$. Hence, being   $\pol^*_{\mathrm B_3} \cong  \pol^*_{\mathrm A_3}$ and $ \pol^*_{\mathrm G_2} \cong  \pol^*_{\mathrm A_2}$, we deduce that the polar root polytopes of types $\mathrm B_3$ and $\mathrm G_2$ are also zonotopes. It turns out that, also in these two cases, the polar root polytope is the zonotope generated by the orbit of a single vector, proportional to a coweight. More precisely, the following result holds.

\begin{pro}
In types $\mathrm B_3$ and $\mathrm G_2$, we have
$$
\begin{array}{lll}
 \pol^*_{\mathrm B_3} &
= & \ZT\(W\cdot \frac{o_3}{2}\)
 \\
 \pol^*_{\mathrm G_2} &
=& \ZT\(W\cdot \frac{o_1}{2}\) 
\end{array}
$$
\end{pro}

\begin{proof}
One inclusion follows by Proposition~\ref{maxZT} since we have:

\par
$\begin{matrix}
{\mathrm B_3}: \hf &\Phi_0=\Phi(\Pi\setminus \{\alpha_2\})\cong \mathrm A_1\times \mathrm A_{1},\hf\\   
&\Phi_0^3=\Phi(\Pi\setminus \{\alpha_2, \alpha_3\})\cong \mathrm A_1,\quad q_3=2,\quad m_3=r_3=2.\hf
\end{matrix}$

\par
$\begin{matrix}
{\mathrm G_2}: &\Phi_0=\Phi(\Pi\setminus \{\alpha_2\})\cong \mathrm A_1,\quad \Phi_0^1=\Phi(\Pi\setminus \{\alpha_1, \alpha_2\})=\emptyset,\quad  q_1=2, 
\quad m_1=r_1=3.
\end{matrix}$

\par
The other inclusion can be directly proved using Lemma~\ref{contiene-contenuto}, (\ref{iuto}).
\end{proof}

\section {Polar root polytopes that are not zonotopes} 
\label{are-not-zone}

In this section, we prove item (4)  of Theorem~\ref{main}, i.e. that, for all root systems $\Phi$ other than those of types $\mathrm A_n$, $\mathrm C_n$, $\mathrm B_3$ and $\mathrm G_2$, the polar root polytope $\pol^*$ is not a zonotope. 

\par
In fact,  we show
that,  for all such root systems, the set of cones on the facets of the root polytope $\pol$ is not equal to the set of closures of the regions of the hyperplane arrangement $\mathcal H_\pol$. 
This is enough to show that $\pol^*$ cannot be a zonotope by Proposition~\ref{prop-zon}, noting that:

\begin{enumerate}
\item being the convex hull of the long roots in $\Phi$, $\pol$ is centrally symmetric with respect to the null vector $\0$, 
\item the polar of a polytope which is  centrally symmetric with respect to $\0$  is centrally symmetric with respect to  the null vector in the dual space,
\item every zonotope which is centrally symmetric with respect to the null vector $\0$ is of the form $\ZN_{\0}(S)$, for an appropriate set $S$.
\end{enumerate}

\par
Recall that $\mathcal H_\pol$ is the central hyperplane arrangement determined by the $(n-2)$-faces of $\pol$, i.e., $H\in \mathcal H_\pol$ if and only if $H$ is a hyperplane containing $\0$ and some $(n-2)$-face of $\pol$. 
We show that some hyperplane in $\mathcal H_\pol$ meets the interior of some facet of $\pol$, for all root types other than $\mathrm A_n$, $\mathrm C_n$, $\mathrm B_3$, $\mathrm G_2$. 
More precisely, for each of the root systems $\Phi$ we are considering, we point out an hyperplane of $\mathcal H_\pol$ containing the barycenter of a facet of $\pol$ and hence cutting that facet. 

\par
For the reader convenience, we recall from \cite{CM1} that $\pol$ has certain  distinguished faces, that we call {\em standard parabolic}, which give a complete set of representatives of the $W$\nobreakdash-\,orbits of faces of $\pol$ (see \cite[Corollary~4.3 and Theorem~5.11]{CM1}). For each $I \subseteq \{1, \ldots, n\}$, we let  
$$
F_I:=\conv\{\alpha\in \Phi^+\mid ( \omega_i^{\vee},  \alpha )= m_i, \text{ for all $i \in I$} \}
$$ 
be the standard parabolic face associated with $I$.
Here we need  the standard parabolic facets, which are all of the form  $F_i:= F_{  \{ i \} }$, for  $i \in \{1, \ldots, n\}$ (not all standard parabolic faces indexed by singletons are facets but all standard parabolic faces which are facets are  indexed by singletons). The numbers $i$ such that $F_i$ are facets are those such that the extended Dynkin diagram is still connected after removing $\alpha_i$ (see \cite[Section~5]{CM1}).  Moreover, we will use the fact that the barycenter of a standard parabolic facet $F_i$ is a multiple of the corresponding fundamental weight $\omega_i$ (see \cite[Lemma~4.2]{CM1}). 

\par
As we already recalled in Section~4,  there exists a subset $H_{\Phi}\subseteq \{1, \ldots, n\}$ such that 
$\mathcal H_{\pol}= \{w(\omega^\vee_k)^\perp \mid w \in W , k \in H_{\Phi}\}$. We call {\em standard hyperplanes} the hyperplanes $(\omega^\vee_k)^\perp$,  $k \in H_{\Phi}$. 

\par
For each irreducible root type  other than $\mathrm A_n$, $\mathrm C_n$, $\mathrm B_3$, $\mathrm G_2$, 
we will exhibit a standard hyperplane containing the barycenter of a facet. 
Since each facet is of the form $wF_i$, and the barycenter of $w F_i$ is a scalar multiple of $w(\omega_i)$, it suffices to find $i,k \in \{1, \dots, n\}$ and $w\in W$  such that: $F_i$ is a facet,  $(\omega_k^\vee)^\perp$ is a standard hyperplane, and 
$w(\omega_i)\perp \omega_k^\vee$. 
In the following table, beside each root type other than $\mathrm A_n$, $\mathrm C_n$, $\mathrm B_3$ and $\mathrm G_2$, in the first row we list all the standard parabolic facets and all the standard hyperplanes; in the further rows we write down explicitly a particular triple  $i,k,w$ such that 
$i,k \in \{1, \dots, n\}$, $w\in W$  and $w(\omega_i)\perp \omega_k^\vee$.

\bigskip

\begin{center}
{{
\renewcommand{\arraystretch}{1.2}
\renewcommand{\tabcolsep}{0.2cm}
\begin{longtable}{|c|l|}
\hline 
$\begin{array}{c}
\mathrm B_n\\
n\geq 4
\end{array} $
& \vbox{\hsize=12cm\noindent
\begin{tabular}{l}
$F_1$, $F_n$, \qquad 
$(\omega_1^\vee)^\perp$, $(\omega_n^\vee)^\perp$\\
$\omega_1=\alpha_1+\cdots+\alpha_n$, \qquad  $s_1(\omega_1)=(\omega_1-\alpha_1)\perp \omega_1^\vee$
\end{tabular}
}\\
\hline 
$\begin{array}{c}
\mathrm D_n\\
n\geq 4
\end{array} $
&
\vbox{\hsize=12cm\noindent
\begin{tabular}{l}
 $F_1$, $F_{n-1}$,  $F_n$, \qquad 
$(\omega_1^\vee)^\perp$, $(\omega_{n-1}^\vee)^\perp$, $(\omega_n^\vee)^\perp$
\\
$\omega_1=\alpha_1+\cdots+\alpha_n$,\qquad 
$ s_1(\omega_1)=(\omega_1-\alpha_1)\perp \omega_1^\vee$
\end{tabular}
}\\
\hline
$\mathrm  E_6$ 
&
\vbox{\hsize=14cm\noindent
\begin{tabular}{l}
 $F_1$, $F_6$, \qquad
$(\omega_1^\vee)^\perp$, $(\omega_{6}^\vee)^\perp$
\\
$\omega_1=\frac{4}{3}\alpha_1+\alpha_2+\frac{5}{3}\alpha_3+2\alpha_4+\frac{4}{3}\alpha_5+\frac{2}{3}\alpha_6$
\\
$s_2s_4s_3s_1(\omega_1)=\omega_1-\alpha_1-\alpha_3-\alpha_4-\alpha_2\perp \omega_2^\vee$
\end{tabular}
}\\
\hline
$\mathrm E_7$ 
&
\vbox{\hsize=14cm\noindent
\begin{tabular}{l}
 $F_2$, $F_7$, \qquad
$(\omega_1^\vee)^\perp$, $(\omega_{2}^\vee)^\perp$
\\
$\omega_7=\alpha_1+\frac{3}{2}\alpha_2+2\alpha_3+3\alpha_4+\frac{5}{2}\alpha_5+2\alpha_6+\frac{3}{2}\alpha_7$
\\
$s_1s_3s_4s_5s_6s_7(\omega_7)=\omega_7-\alpha_7-\alpha_6-\alpha_5
-\alpha_4-\alpha_3-\alpha_1\perp \omega_1^\vee$
\end{tabular}
}\\
\hline
$\mathrm  E_8$ 
&
\vbox{\hsize=14cm\noindent
\begin{tabular}{l}
 $F_1$, $F_2$, \qquad
$(\omega_2^\vee)^\perp$, $(\omega_{8}^\vee)^\perp$
\\
$\omega_1=4\alpha_1+5\alpha_2+7\alpha_3+10\alpha_4+8\alpha_5+
6\alpha_6+4\alpha_7+2\alpha_8$\\
$s_8s_7s_6s_5s_4s_3s_1(s_2s_4s_3s_5s_4s_2s_6s_5s_4s_3)s_1(\omega_1)=$\\
$s_8s_7s_6s_5s_4s_3s_1(s_2s_4s_3s_5s_4s_2s_6s_5s_4s_3)
(\omega_1-\alpha_1)=s_8s_7s_6s_5s_4s_3s_1(\omega_1-\theta_6)=$\\
$(\omega_1-
\alpha_1+\alpha_3+\alpha_4+\alpha_5+\alpha_6+\alpha_7+\alpha_8)-
(\theta_6-\alpha_7-\alpha_8)\perp\omega_8^\vee$
\\
$[${\it here $\theta_6$ is the highest root of the type $E_6$ root system generated by $\alpha_1\dots, \alpha_6$}$]$
\end{tabular}
}\\
\hline
$\mathrm F_4$ 
&
\vbox{\hsize=14cm\noindent
\begin{tabular}{l}
$F_4$, \qquad $(\omega_4^\vee)^\perp$ 
\\
$\omega_4=\alpha_1+2\alpha_2+3\alpha_3+2\alpha_4$
\\
$s_4s_3(s_2s_3s_4)(\omega_4)=s_4s_3(\omega_4-\alpha_4-\alpha_3-\alpha_2)=
\omega_4-2\alpha_4-2\alpha_3-\alpha_2\perp\omega_4^\vee$
\end{tabular}
}\\
\hline
\end{longtable}
}}
\end{center}

\end{document}